\documentclass[12pt]{amsart}
\usepackage[T1]{fontenc}
\usepackage{newtxtext,newtxmath}
\usepackage{graphicx,mathrsfs,setspace,fancyhdr,amsmath,fullpage,amssymb}
\usepackage[all]{xy}
\usepackage{framed}
\usepackage[margin=2.5cm]{geometry}
\newtheorem{thm}{Theorem}[section]
\newtheorem*{thma}{Theorem A}
\newtheorem*{thmB}{Theorem B}
\newtheorem{defprop}[thm]{Definition-Proposition}

\newtheorem{cor}[thm]{Corollary}

\newtheorem{lem}[thm]{Lemma}
\newtheorem{prop}[thm]{Proposition}

\newtheorem*{propC}{Proposition C}
\theoremstyle{definition}
\newtheorem{defn}[thm]{Definition}
\newtheorem{rem}[thm]{Remark}
\newtheorem{rmk}[thm]{Remark}
\numberwithin{equation}{section}

\newcommand{\eps}{\varepsilon}

\makeatletter
\newcommand\xleftrightarrow[2][]{\ext@arrow 0099{\longleftrightarrowfill@}{#1}{#2}}
\def\longleftrightarrowfill@{\arrowfill@\leftarrow\relbar\rightarrow}
\makeatother

\def\mc{\mathcal}

\DeclareMathOperator{\Pic}{Pic}

\def\N{\mathbb{N}}

\def\Q{\mathbb{Q}}
\def\R{\mathbb{R}}

\def\Z{\mathbb{Z}}

\def\cD{\mc{D}}
\def\cE{\mc{E}}
\def\cF{\mc{F}}

\def\cH{\mc{H}}

\def\cL{\mc{L}}

\def\O{\mc{O}}
\def\cO{\mc{O}}
\def\cP{\mc{P}}

\def\cU{\mc{U}}

\newcommand{\reg}[2]{\ensuremath{{\rm reg}(#2,#1)}}
\newcommand{\creg}[2]{\ensuremath{{\rm reg}_{\rm cont}(#2,#1)}}
\newcommand{\rego}[1]{\reg{\cO(1)}{#1}}
\newcommand{\crego}[1]{\creg{\cO(1)}{#1}}
\newcommand{\deq}{\ensuremath{\stackrel{\rm def}{=}}}
\newcommand{\st}[1]{\ensuremath{\left\{ #1 \right\}}}

\begin{document}
\title{Continuous CM-Regularity of Semihomogeneous Vector Bundles}
\author{Alex K\"{u}ronya, Yusuf Mustopa}
\begin{abstract}
We prove that the continuous CM (Castelnuovo-Mumford) regularity of a semihomogeneous vector bundle on an abelian variety $X$ is a piecewise-constant function of its rational Chern data.  We also use generic vanishing theory to obtain a sharp upper bound on the continuous CM regularity of any vector bundle on $X.$  Equality in this bound is attained for certain semihomogeneous bundles, including many Verlinde bundles when $X$ is a Jacobian.
\end{abstract}
\maketitle

\section*{Introduction}
 
If $X$ is a smooth projective variety and $\cO(1)$ is an ample and globally generated line bundle on $X,$ the CM (Castelnuovo-Mumford) regularity of a coherent sheaf $\cF$ with respect 
to $\cO(1)$ is  
\[
 \reg{\cO(1)}{\cF} \deq \min\st{ m\in\Z\mid \forall\,i>0\ \ H^i(\cF(m-i))=0} \ 
\]
This invariant is instrumental in measuring the homological complexity of the graded module $H^{0}_{\ast}(\cF) = \oplus_{m}H^{0}(\cF(m))$, and at the same time  
helps measure positivity,  since $\cF(m)$ is globally generated for all $m \geq \reg{\cO(1)}{\cF}.$

In light of Koll\'{a}r's suggestion that the positivity of adjoint bundles should be controlled by intersections of Chern classes \cite{Kol} we may ask the following:  
if $\cE$ is a vector bundle on $X$, when does ${\rm reg}(\omega_{X} \otimes \cE,\cO(1))$ depend entirely on the Chern data of $\cE,$ $\cO(1)$ and $X$?  
The purpose of this note is to address this question when $X$ is an abelian variety.  

An example of the numerical nature of positivity on abelian varieties can be seen in \cite{KL}, where it was shown that the $N_p$ property for an embedding of an abelian surface $X$ is ``almost numerical" in that when the degree of $X$ is sufficiently large, the obstruction to the $N_p$ property is an elliptic curve of low degree on $X.$   

A variant of CM regularity well-suited to sheaves on irregular varieties was introduced and studied in \cite{Mus}.  The \textit{continuous CM regularity} of a coherent sheaf $\cF$ with respect to $\cO(1)$, which is denoted by $\crego{\cF},$ is equal to ${\rm reg}(\cF \otimes \alpha,\cO(1))$ for general $\alpha \in \widehat{X}$; see Section \ref{subsec:pos-irreg} for details.  A related notion, that of continuous rank, has recently been studied in \cite{BPS}.  We will study continuous CM regularity for the simplest interesting class of bundles on abelian varieties.

Semihomogeneous bundles were first studied by Mukai in \cite{Mu1} and are precisely the vector bundles on an abelian variety which admit an Atiyah-type classification; they include all line bundles on abelian varieties, as well as all indecomposable vector bundles on elliptic curves.  Among their many pleasant properties is that the Chern classes of a semihomogeneous bundle $\cE$ of rank $r$ all depend on $c_{1}(\cE)/r$ (e.g.~ Theorem 5.12 of \cite{Y}).  Our first main result establishes that continuous CM regularity is a numerical property for semihomogeneous vector bundles. 

\begin{thma}
Let $X$ be an abelian variety, and let $\eta =c_{1}(\cO(1)).$  Then there exists a piecewise-constant function $\rho_{\eta} : N^{1}_{\R}(X) \to \Z$ satisfying the property that for any semihomogeneous bundle $\cE$ of rank $r$ on $X$, we have ${\rm reg}_{\rm cont}(\cE,\cO(1)) = \rho_{\eta}(c_{1}(\cE)/r).$
\end{thma}

It is important to note that we cannot replace ${\rm reg}_{\rm cont}(\cE,\cO(1))$ by ${\rm reg}(\cE,\cO(1))$ in this statement (see Remark \ref{rmk:reg-ineq}).  In addition, the function $\rho_{\eta}$ is almost completely determined by continuous CM-regularity, as every rational point in $N^{1}_{\R}(X)$ is of the form $c_{1}(\cE)/r$ for some positive integer $r$ and some semihomogeneous bundle $\cE$ of rank $r$ on $X$ (Proposition 6.22 of \cite{Mu1}).  Since the surjectivity of certain multiplication maps is an important consequence of CM-regularity, we should mention the numerical criterion for the surjectivity of multiplication maps associated to semihomogeneous bundles given in Theorem 6.13 of [PP3].

The proof of Theorem A, which is given in Section \ref{subsec:thm-a}, has two main steps.  First we reduce to the case where $\cE$ is simple; although simple semihomogeneous bundles are widely known to be the building blocks of general semihomogeneous bundles \cite{Mu1} this reduction is far from immediate, and is accomplished by Propositions \ref{lem:sh-nondeg} and \ref{prop:simp-uni}.  One crucial observation we rely on here is Proposition 6.3 of \cite{Gul} (restated here as Proposition \ref{prop:sh-it}) which says that a simple semihomogeneous bundle is a WIT-sheaf, and is an IT-sheaf when it is nondegenerate.

The next step is to define the function $\rho_{\eta}$ and show that ${\rm reg}_{\rm cont}(\cE,\cO(1)) = \rho_{\eta}(c_{1}(\cE)/r)$ when $\cE$ is simple and semihomogeneous.  To this end we use the index function for nondegenerate elements of $N^{1}_{\R}(X)$, which was studied by Grieve in \cite{Gr}.  While the definition of $\rho_{\eta}$ is a bit involved, its complexity is governed by the chamber structure on $N^{1}_{\R}(X)$ coming from the index function.

It is natural to consider how the continuous CM-regularity of semihomogeneous bundles on $X$ compares to that of general vector bundles.  Our second main result implies that when the first Chern class of a semihomogeneous bundle is proportional to $c_{1}(\cO(1)),$ its continuous CM-regularity is extremal.  

\begin{thmB}
If $X$ is an abelian variety of dimension $g \geq 1$ and $\cE$ is a vector bundle on $X,$
\begin{equation}
\label{eq:creg-vb}
{\rm reg}_{\rm cont}(\cE,\cO(1)) \leq \min\{m \in \Z : \cE(m-g)\textnormal{ is a GV-sheaf}\}
\end{equation}
and equality is attained if $\cE$ is semihomogeneous and $c_{1}(\cE) \in {\Q}c_{1}(\cO(1)).$
\end{thmB}

The proof makes essential use of the tensorial properties of GV-sheaves developed in \cite{PP3}, and is contained in Section \ref{subsec:prop-b}; see Proposition \ref{prop:gen-ineq} for a precise description of the case of equality in (\ref{eq:creg-vb}).  

Returning to positivity considerations, Theorem 4.1 of \textit{loc.~ cit.} says that every GV-sheaf on an abelian variety is nef, and it is therefore natural to ask which GV-sheaves have minimal positivity.  The next result helps quantify this in terms of continuous CM-regularity.  In spite of the superficial similarity to Theorem B, the first Chern classes of the semihomogeneous bundles attaining equality are not assumed proportional to $c_{1}(\cO(1)).$

\begin{propC}
In the setting of Theorem B, if the vector bundle $\cE$ is a GV-sheaf, then 
\begin{equation}
\label{eq:gv-ineq}
{\rm reg}_{\rm cont}(\cE,\cO(1)) \leq g
\end{equation}
and equality is attained if $\cE$ is semihomogeneous and $c_{1}(\cE^{\vee}(1))$ is ample.
\end{propC}

If the GV hypothesis on $\cE$ is strengthened slightly, we get a correspondingly stronger statement (Proposition \ref{prop:it-index}).

When $X$ is a Jacobian, Theorem 1 of \cite{O} implies that the Verlinde bundle $\mathbf{E}_{r,k}$ for a rank-level pair $(r,k)$ with ${\rm gcd}(r,k)$ odd is semihomogeneous and ample, and is therefore a semihomogeneous GV-sheaf by Proposition \ref{prop:semihom-nef}.  In Section \ref{sec:verlinde} we calculate the continuous CM regularity of such an $\mathbf{E}_{r,k}$ with respect to a pluri-theta polarization; this attains equality in (\ref{eq:gv-ineq}) when $k/r$ is not too large (Proposition \ref{prop:verlinde}). 

In the last part of our paper (Section \ref{sec:products}) we study the continuous CM-regularity of vector bundles on products of non-isogenous abelian varieties, and provide an explicit formula in the case of two non-isogenous elliptic curves. 

\subsection*{Acknowledgements} This research started while the authors participated in the Combinatorial Algebraic Geometry Major Thematic Program at the Fields Institute; we thank its organizers for the motivating atmosphere.  We would also like to thank Mihnea Popa for useful discussions and helpful comments on an earlier draft. 

\subsection*{Notation and conventions} We work over the complex numbers.  Throughout, $X$ is a smooth projective variety and $\cO(1)$ is an ample and globally generated 
line bundle on $X.$

\section{Preliminaries}

\subsection{Positivity of Sheaves on Irregular Varieties} 
\label{subsec:pos-irreg}

We begin by introducing an object central to the study of coherent sheaves on irregular varieties.

\begin{defn}
If $\cF$ is a coherent sheaf on $X$ and $i \geq 0$, the $i-$th cohomological support locus of $\cF$ is
\[
 V^i(\cF) \deq \st{\alpha\in \Pic^0(X)\mid H^i(\cF\otimes\alpha)\neq 0}\ \ \ 
\]
For $i < 0$ we define $V^i(\cF) = \emptyset.$
\end{defn}

The following notion first appeared in \cite{PP1}, and was shown to imply ampleness in \cite{De}, solidifying the connection between the loci $V^{i}(\cF)$ and the positivity of $\cF.$

\begin{defn}
$\cF$ is $M-$regular if ${\rm codim}(V^{i}(\cF)) > i$ for all $i > 0.$
\end{defn}

Much of the power of $M-$regularity comes from \cite[Proposition 2.13]{PP1}, which says that $M-$regular sheaves satisfy the following property.

\begin{defn}
	$\cF$ is \textit{continuously globally generated} if there is a nonempty Zariski-open subset $U \subseteq {\rm Pic}^{0}(X)$ such that the evaluation map
	\begin{equation}
	\bigoplus_{\alpha \in U}H^{0}(\cF \otimes \alpha) \otimes \alpha^{\vee} \to \cF
	\end{equation}
	is surjective. 
\end{defn}

The next definition is slightly weaker than $M-$regularity, and yet plays a more important role in this paper; it was studied systematically in \cite{PP2,PP3}.  


\begin{defn}
$\cF$ is a \textit{GV-sheaf} if ${\rm codim}(V^{i}(\cF)) \geq i$ for all $i > 0.$
\end{defn}

The main theorem of \cite{GL} implies $\omega_{X}$ is a GV-sheaf when $X$ is of maximal Albanese dimension.  The ampleness of $M-$regular sheaves was later used in \cite{PP3} to prove that GV-sheaves on abelian varieties are all nef.

Continuous CM-regularity is a notion inspired by \cite{PP1} and is adapted to coherent sheaves on irregular varieties in that it links homological complexity to cohomological support loci.  For its basic properties we refer the reader to \cite{Mus}. 

\begin{defn}
\label{def:cont-reg}
The \emph{continuous CM-regularity} of $\cF$ with respect to $\cO(1)$ is 
 \[
  \crego{\cF} \deq \min\st{m\in\Z\mid \forall i>0\ \ V^i(\cF(m-i)) \neq \Pic^0(X)} \ 
 \]
\end{defn}

Note that $\crego{\cF} = \rego{\cF \otimes \alpha}$ for general $\alpha \in {\rm Pic}^{0}(X).$

\begin{rmk}
\label{rmk:reg-ineq}
In general we have that $\crego{\cF} \leq \rego{\cF}.$  To see that strict inequality is possible, observe that $\rego{\cO(1)}=g$ and $\crego{\cO(1)}=g-1.$
\end{rmk}

\subsection{IT- and WIT-sheaves}

We will assume for the rest of the paper that $X$ is an abelian variety.  In the next definition, $\cP$ is the normalized Poincar\'{e} bundle on $X \times \widehat{X}$ and $p,\widehat{p}$ are the projections from $X \times \widehat{X}$ to $X$ and $\widehat{X},$ respectively.

\begin{defn}
$\cF$ is a WIT-sheaf if $R^{i}\widehat{p}_{\ast}(p^{\ast}\cF \otimes \cP) \neq 0$ for at most one $i \in \{0, \cdots ,g\}.$  If $\cF$ is a WIT-sheaf and $R^{j}\widehat{p}_{\ast}(p^{\ast}\cF \otimes \cP) = 0$ for all $j \neq i$ we say that $i$ is the \textit{index} of $\cF$; this is denoted by ${\iota}(\cF).$
\end{defn}

The following lemma is important for the proof of Theorem A.

\begin{lem}
\label{lem:wit-deg}
If $\cF$ is a WIT-sheaf satisfying $\chi(\cF)=0,$ then $V^{j}(\cF) \neq \widehat{X}$ for all $j \geq 0.$
\end{lem}

\begin{proof}
Let $i={\iota}(\cF).$  Since $V^{j}(\cF) \neq \widehat{X}$ for $j \neq i$ it is enough to show that $V^{i}(\cF) \neq \widehat{X}$ where $i$ is the index of $\cF.$  If $\widehat{\cF} \in {\rm Coh}(\widehat{X})$ is the unshifted Fourier transform of $\cF,$ then by our hypothesis and Proposition 14.7.7 of \cite{BL} we have ${\rm rank}(\widehat{\cF}) = (-1)^{i}\chi(\cF) = 0.$  By Grauert's Theorem $H^{i}(\cF \otimes \alpha) =0$ for general $\alpha \in \widehat{X},$ i.e.~ $V^{i}(\cF) \neq \widehat{X}$ as desired.
 \end{proof}
 
\begin{defn}
$\cF$ is an IT-sheaf if $V^{i}(\cF) \neq \emptyset$ for at most one $i \in \{0, \cdots ,g\}.$  
\end{defn}

Since a nonzero IT-sheaf $\cF$ is a WIT-sheaf, we may speak of its index ${\iota}(\cF)$, which is the unique $i$ for which $V^{i}(\cF) \neq \emptyset.$  Note that any IT-sheaf of index 0 is $M-$regular (and in particular, a GV-sheaf).  The following result will be used in the proofs of Theorem B and Corollary C.

\begin{prop}
	\label{prop:mreg-it}
	\cite[Proposition 3.1]{PP3} If $X$ is an abelian variety, $\cF$ is a GV-sheaf on $X,$ and $\cH$ is a locally free sheaf on $X$ which is I.T. of index 0, then $\cF \otimes \cH$ is I.T. of index 0. \hfill \qedsymbol
\end{prop} 





\subsection{Semihomogeneous Bundles}

Semihomogeneous vector bundles are important building blocks of coherent sheaves generalizing line bundles and indecomposable vector bundles on elliptic curves at the same time. Here we summarize the information we will need about them. 

\begin{defn}
A vector bundle $\cE$ on $X$ is \textit{semihomogeneous} if for all $x \in X$ there exists $\alpha \in \widehat{X}$ such that $t_{x}^{\ast}\cE \cong \cE \otimes \alpha.$
\end{defn}


\begin{prop}
\label{prop:semi-decomp}
Let $\cE$ be a semihomogeneous bundle on $X.$  Then there exist simple semihomogeneous bundles $\cE_{1}, \cdots \cE_{s}$ and indecomposable unipotent bundles $\cU_{1}, \cdots ,\cU_{s}$ on $X$ such that
\begin{equation}
\label{eq:sh-decomp}
\cE \cong \bigoplus_{j=1}^{s}\cE_{j} \otimes \cU_{j}
\end{equation}
\end{prop}

\begin{prop}
\label{prop:semihom-dir}
(Theorem 5.8, \cite{Mu1}) Let $\cE$ be a simple vector bundle on $X.$  Then $\cE$ is semihomogeneous if and only if there exists an abelian variety $Z,$ an isogeny $p : Z \to X$ and a line bundle $\cL$ on $Z$ such that $\cE \cong p_{\ast}\cL.$
\end{prop}

The following statement is Proposition 6.2 of \cite{Gul}.  Since we will be referring to different parts of it throughout different proofs, we list it here for reference; see \textit{loc.~ cit.~} for the proof.

\begin{prop}
\label{prop:sh-it}
If $\cE$ is a simple semihomogeneous bundle on $X$, the following properties hold:
\begin{itemize}
\item[(i)]{$\cE$ is a WIT-sheaf, and ${\iota}(\cE)={\iota}(\det(\cE)).$}
\item[(ii)]{If $\chi(\cE) \neq 0,$ then $\cE$ is an IT-sheaf. \hfill \qedsymbol}
\item[(iii)]{$\chi(\cE) \neq 0$ if and only if $\chi(\det(\cE)) \neq 0.$}
\end{itemize}
\end{prop}
 
\section{Proofs of Main Results}

\subsection{Proof of Theorem A}
\label{subsec:thm-a}

Here we establish the central result of the paper saying that continuous regularity is a numerical property for semihomogeneous vector bundles. 

We begin with a reduction step to the case of simple bundles.

\begin{prop}
\label{lem:sh-nondeg}
Let $\cE$ be a simple semihomogeneous bundle on $X$ and let $\cU$ be a unipotent bundle on $X$.  Then for all $i \geq 0$ we have that $V^{i}(\cE) = \widehat{X}$ if and only if $V^{i}(\cE \otimes \cU) = \widehat{X}.$
\end{prop}

\begin{proof}
We proceed by induction on $r := {\rm rank}(\cU).$  The case $r = 1$ is immediate.  If $r \geq 2,$ then $\cU$ can be expressed as an extension
\begin{equation}
0 \to \cO_{X} \to \cU \to {\cU}' \to 0
\end{equation}
where  ${\cU}'$ is a unipotent bundle of rank $r-1$ on $X.$  Fixing $i \geq 0$ we have
\begin{equation}
\label{eq:sh-incl}
V^{i}(\cE)-V^{i-1}(\cE \otimes {\cU}') \subseteq V^{i}(\cE \otimes \cU) \subseteq V^{i}(\cE) \cup V^{i}(\cE \otimes {\cU}')
\end{equation}
Assume $V^{i}(\cE)=\widehat{X}$.  Since $\cE$ is a WIT-sheaf by (i) of Proposition \ref{prop:sh-it} we have from Lemma \ref{lem:wit-deg} that $\chi(\cE) \neq 0;$ consequently $\cE$ is IT of index $i$ by (ii) of Proposition \ref{prop:sh-it}.  In particular, $V^{i-1}(\cE) = \emptyset,$ so by our inductive hypothesis $V^{i-1}(\cE \otimes {\cU}') \neq \widehat{X}.$  The first inclusion in (\ref{eq:sh-incl}) then implies that $V^{i}(\cE)-V^{i-1}(\cE \otimes {\cU}') = \widehat{X}-V^{i-1}(\cE \otimes {\cU}')$ is a nonempty Zariski-open subset of $\widehat{X}$ contained in the Zariski-closed set $V^{i}(\cE \otimes \cU).$  Therefore $V^{i}(\cE \otimes \cU)=\widehat{X}.$   

Conversely, if $V^{i}(\cE \otimes \cU) = \widehat{X}$, the second inclusion in (\ref{eq:sh-incl}) and our inductive hypothesis imply that $V^{i}(\cE) = V^{i}(\cE \otimes {\cU}') = \widehat{X}.$
\end{proof}

Our next result shows how to relate the continous CM-regularity of a semihomogeneous bundle 
to essentially that of a direct summand. 

\begin{prop}
\label{prop:simp-uni}
Let $\cE$ be a semihomogeneous bundle on $X$ and let $\cE' \otimes \cU$ be an indecomposable direct summand of $\cE$, where ${\cE}'$ is simple semihomogeneous and $\cU$ is unipotent.  Then the following hold:
\begin{itemize}
\item[(i)]{$c_{1}(\cE)/{\rm rank}(\cE) = c_{1}({\cE}')/{\rm rank}({\cE}').$}
\item[(ii)]{${\rm reg}_{\rm cont}(\cE,\cO(1))={\rm reg}_{\rm cont}({\cE}',\cO(1)).$}
\end{itemize}

\end{prop}

\begin{proof}
Consider the decomposition $\cE \cong \oplus_{j=1}^{s}\cE_{j} \otimes \cU_{j}$ guaranteed by Proposition \ref{prop:semi-decomp}.  Without loss of generality we may take $\cE' = {\cE}_{1}.$  By Propositions 6.9 and 6.15 of \cite{Mu1} we have for $1 \leq j \leq s$ that 
\begin{equation}
\label{eq:slopes}
c_{1}(\cE)/{\rm rank}(\cE) = c_{1}(\cE_{j} \otimes \cU_{j})/{\rm rank}(\cE_{j} \otimes \cU_{j}) = c_{1}(\cE_{j})/{\rm rank}(\cE_{j})
\end{equation}
so that (i) is proved.  Turning to (ii), we have for all $i > 0$ and all $m \in \Z$ that
\begin{equation}
\label{eq:supp-eq}
V^{i}(\cE(m-i)) = \displaystyle\bigcup_{j=1}^{s}V^{i}(\cE_{j}(m-i) \otimes \cU_{j})
\end{equation}
By (\ref{eq:slopes}) and Proposition 6.17 of \cite{Mu1} we have that for $1 \leq j \leq s$ there exists $\alpha_{j} \in \widehat{X}$ such that $\cE_{1} \cong \cE_{j} \otimes \alpha_{j}.$  It follows that the $V^{i}(\cE_{j}(m-i))$ are all translates of $V^{i}(\cE_{1}(m-i))$, and we may conclude from (\ref{eq:supp-eq}) and Proposition \ref{lem:sh-nondeg} that for all $i > 0$ we have $V^{i}(\cE(m-i)) \neq \widehat{X}$ if and only if $V^{i}(\cE_{1}(m-i)) \neq \widehat{X}.$  Consequently ${\rm reg}_{\rm cont}(\cE,\cO(1)) = {\rm reg}_{\rm cont}(\cE_{1},\cO(1))$ as desired.
\end{proof}

We now turn to the other main component of the proof of Theorem A, namely the index function studied in \cite{Gr}.

\begin{defprop}
\label{defprop:index}
If $U = \{{\gamma} \in N^{1}_{\R}(X) : \gamma^{\cdot g} \neq 0\}$ there is a unique function $\iota : U \to \{0, \cdots ,g\}$ which satisfies the following properties:
\begin{itemize}
\item[(i)]{$\iota$ is constant on each connected component of $U.$}
\item[(ii)]{For any simple semihomogeneous bundle $\cE$ of rank $r \geq 1$ on $X$ which is nondegenerate, we have that $\iota(c_{1}(\cE)/r)$ is the index of $\cE.$}
\end{itemize}
\end{defprop}

\begin{proof}
Recall that for any nondegenerate line bundle $L$ on $X$, the index of $L$ depends only on $c_{1}(L),$ so that we may speak of the index of any element of $N^{1}_{\Z}(X).$  By Corollary 4.2 of \cite{Gr}, there exists a function $\iota : U \to \{0, \cdots ,g\}$ which is constant on each connected component of $U$ (therefore satisfying (i)) and has the property for all $\eta \in N^{1}_{\Q}(X)$ and $a \in \N$ satisfying $a{\eta} \in N^{1}_{\Z}(X),$ we have ${\iota}(\eta)$ is the index of $a{\eta}.$

Let $\cE$ be a simple semihomogeneous bundle on $X$ which is nondegenerate.  Then by Proposition \ref{prop:sh-it} we have that $\cE$ is an IT-sheaf and ${\iota}(c_{1}(\cE)/r)$ is equal to the index of $\det{\cE},$ which in turn is the index of $\cE.$  
\end{proof}

\begin{prop}
Let $\eta \in N^{1}_{\R}(X)$ be an ample class.  Then the function ${\rho}_{\eta} : N^{1}_{\R}(X) \to \Z$ defined by
\begin{equation}
\label{eq:rho-def}
{\rho}_{\eta}(\gamma) = \min\{m \in \Z : {\forall}i \in \{1, \cdots ,g\}~ (\gamma+(m-i){\eta})^{\cdot g} = 0 \textnormal{ or }{\iota}(\gamma+(m-i){\eta}) \neq i\}
\end{equation}
is well-defined.  Moreover, it is locally constant on each connected component of the set
\begin{equation}
U_{\eta} := \{\gamma \in N^{1}_{\R}(X) : {\forall}t \in {\Z}~ (\gamma + t{\eta})^{\cdot g} \neq 0 \}
\end{equation}
which is open and dense in $N^{1}_{\R}(X)$ with respect to the classical topology.   
\end{prop}

\begin{proof}
To show that $\rho_{\eta}$ is well-defined it is enough to check that for all $\gamma \in N^{1}_{\R}(X)$ the set of integers on the right-hand side of (\ref{eq:rho-def}) is nonempty and bounded from below.  Let $\gamma \in N^{1}_{\R}(X)$ be given.  Since $\eta$ is ample, we have that for $m >> 0$ the class $\gamma+(m-i)\eta$ is ample for each $i \in \{1, \cdots ,g\}$ and therefore nondegenerate of index 0; this proves nonemptiness.  Also, for $m' >> 0$ we have that $-\gamma + (m'+g){\eta}$ is ample, so that $\gamma+(-m'-g){\eta}$ is nondegenerate of index $g;$ this proves boundedness from below.  Thus $\rho_{\eta}$ is well-defined as claimed.

Since $N^{1}_{\R}(X)-U_{\eta}$ is the union of all translates of the real-algebraic hypersurface $\{\gamma \in N^{1}_{\R}(X) : \gamma^{\cdot g} = 0\}$ by integer multiples of $\eta$, its complement $U_{\eta}$ is open and dense in $N^{1}_{\R}(X)$ with respect to the classical topology.  

Let $\gamma \in U_{\eta}$ be given, and fix a norm $|{\cdot}|$ on $N^{1}_{\R}(X)$.  By (i) of Definition-Proposition \ref{defprop:index} there exists $\eps > 0$ such that for all ${\gamma}' \in N^{1}_{\R}(X)$ satisfying $|\gamma - {\gamma}'| < \eps$ we have that ${\gamma}' \in U_{\eta}$ and ${\iota}({\gamma}'+(\rho_{\eta}(\gamma)-i){\eta}) \neq i$ for $1 \leq i \leq g.$  Therefore $\rho_{\eta}({\gamma}') \leq \rho_{\eta}(\gamma)$ whenever $|\gamma - {\gamma}'| < \eps.$  Also, there exists $i' \in \{1, \cdots ,g\}$ such that ${\iota}(\gamma+(\rho_{\eta}(\gamma)-1-i')\eta) = i'.$  Applying (i) of Definition-Proposition \ref{defprop:index} once more, there exists ${\eps}' > 0$ such that ${\gamma}' \in U_{\eta}$ and ${\iota}({\gamma}'+(\rho_{\eta}(\gamma)-1-i')\eta) = i'$ whenever $|\gamma - {\gamma}'| < {\eps}'.$  We may conclude that $\rho_{\eta}(\gamma) = \rho_{\eta}({\gamma}')$ whenever $|\gamma - {\gamma}'| < \min\{\eps,{\eps}'\},$ e.g.~ that $\rho_{\eta}$ is constant on an open neighborhood of $\gamma$ in $U_{\eta}.$  It follows that $\rho_{\eta}$ is constant on the connected component of $U_{\eta}$ which contains $\gamma.$
\end{proof}

The following result concludes the proof of Theorem A.

\begin{prop}
\label{prop:exact-reg}
If $\cE$ is a semihomogeneous bundle of rank $r \geq 1$ on $X$ and $\eta = c_{1}(\cO(1))$ then ${\rm reg}_{\rm cont}(\cE,\cO(1)) = {\rho}_{\eta}(c_{1}(\cE)/r).$
\end{prop}

\begin{proof}
By Proposition \ref{prop:simp-uni} we may assume without loss of generality that $\cE$ is simple; in particular, $\cE(t)$ is simple for all $t \in \Z.$  For each $m \geq \crego{\cE}$ and $i \in \{1, \cdots ,g\}$ we have from Proposition \ref{prop:sh-it} that if $\cE(m-i)$ is nondegenerate, then
\begin{equation}
{\iota}\biggl(\frac{c_{1}(\cE)}{r}+(m-i)\eta\biggr) = {\iota}\biggl(\frac{c_{1}(\cE(m-i))}{r}\biggr) = {\iota}(c_{1}(\cE(m-i))) = {\iota}(\cE(m-i)) \neq i
\end{equation}
Consequently $\rho_{\eta}(c_{1}(\cE)/r) \leq \crego{\cE}.$  If $\rho_{\eta}(c_{1}(\cE)/r) < \crego{\cE}$ then for some $i' \in \{1, \cdots ,g\}$ we have $V^{i'}(\cE(\rho_{\eta}(c_{1}(\cE)/r)-i')) = \widehat{X}.$  By Lemma \ref{lem:wit-deg} and Proposition \ref{prop:sh-it}, it follows that $\cE(\rho_{\eta}(c_{1}(\cE)/r)-i')$ is nondegenerate and IT of index $i',$ which is impossible.
\end{proof}

\subsection{Proofs of Theorem B and Proposition C}
\label{subsec:prop-b}

Our first task is to establish that most reasonable positivity conditions coincide for semihomogeneous vector bundles.  While the results are almost certainly known to experts, we include their proofs for the sake of completeness.  

 
\begin{prop}
\label{prop:semihom-pos}
Let $\cE$ be a semihomogeneous vector bundle on $X.$  Then the following are equivalent:
\begin{itemize}
	\item[(i)]{$V^{0}(\cE) = \widehat{X}$.}
	\item[(ii)]{$\cE$ is I.T. of index 0.}
	\item[(iii)]{$\cE$ is $M-$regular.}
	\item[(iv)]{$\cE$ is continuously globally generated.}
	\item[(v)]{$\cE$ is ample.}
	\item[(vi)]{${\det}(\cE)$ is ample.}
\end{itemize}
\end{prop}

\begin{proof}
	For (i)$\Rightarrow$(ii), assume that $V^{0}(\cE)=\widehat{X}.$  Since $\cE$ is semihomogeneous and the property of being I.T. of index 0 is preserved under extension, we can assume by Propositions \ref{prop:semi-decomp} and \ref{lem:sh-nondeg} that $\cE$ is simple.  Then $\cE$ is a WIT-sheaf by (i) of Proposition \ref{prop:sh-it}, and combining (ii) of the latter with Lemma \ref{lem:wit-deg} and our assumption on $V^{0}(\cE)$ shows that $\cE$ is I.T. of index 0.  The statement (ii)$\Rightarrow$(i) is immediate, so we have just shown the equivalence (i)$\Leftrightarrow$(ii).    
	
	The implication (ii) $\Rightarrow$ (iii) is immediate from the definitions, and (iii) $\Rightarrow$ (iv) follows from \cite[Proposition 2.13]{PP1}, whereas (iv) $\Rightarrow$(v) follows from Proposition 3.1 of \cite{De}.  Also, (v)$\Rightarrow$(vi) holds since determinants of ample bundles are ample.
	
	It remains to verify (vi)$\Rightarrow$(i), or equivalently (vi)$\Rightarrow$(ii).  Assume $\det(\cE)$ is ample.  By Proposition \ref{prop:semi-decomp} and (i) of Proposition \ref{prop:simp-uni}, $\det(\cE)$ is a tensor product of determinants of simple semihomogeneous bundles whose first Chern classes are positive rational multiples of one another; in particular these determinants are all ample.  Since being I.T. of index 0 is preserved under extension, we can assume by Proposition \ref{prop:semi-decomp} that $\cE$ is simple.  Then $\cE$ is a WIT-sheaf of index 0 by (i) of Proposition \ref{prop:sh-it} and the ampleness of $\det(\cE).$  Since ample line bundles on abelian varieties are nondegenerate, we have from (ii) and (iii) of Proposition \ref{prop:sh-it} that $\cE$ is I.T. of index 0 as desired.  

\end{proof}

\begin{prop}
\label{prop:semihom-nef}
Let $\cE$ be a semihomogeneous vector bundle on $X.$  Then the following are equivalent:
\begin{itemize}
\item[(i)]{$\cE$ is a GV-sheaf.}
\item[(ii)]{$\cE$ is nef.}
\item[(iii)]{$\det(\cE)$ is nef.}
\end{itemize}
\end{prop}

\begin{proof}
The implication (i) $\Rightarrow$ (ii) follows from Theorem 4.1 of \cite{PP3}, and the equivalence (ii) $\Leftrightarrow$ (iii) follows from Proposition 6.2 of \cite{Gul}.  Arguing as in the proof of Proposition \ref{prop:semihom-pos}, we can assume without loss of generality that $\cE$ is simple.  For an isogeny $p : Z \to X$ and a line bundle $\cL$ on $Z$ there is an isomorphism $\cE \cong p_{\ast}\cL.$  It follows from the projection formula that the dual isogeny $p^{\ast} : \widehat{X} \to \widehat{Z}$ maps $V^{i}(\cE)$ onto $V^{i}(\cL)$ for all $i$; therefore $\cE$ is a GV-sheaf if and only if $\cL$ is a GV-sheaf.  We have that
\begin{equation}
p^{\ast}\cE \cong p^{\ast}p_{\ast}\cL \cong \bigoplus_{\eta \in {\rm ker}(\pi)}\cL \otimes \eta
\end{equation}
In particular, $p^{\ast}\det(\cE) \cong \det(p^{\ast}\cE)$ is numerically equivalent to a positive tensor power of $\cL$, and since $\det(\cE)$ is nef, the same is true of $\cL.$  By Corollary C of \cite{PP2} any nef line bundle on an abelian variety is a GV-sheaf; consequently $\cL$ is a GV-sheaf, and therefore so is $\cE.$
\end{proof}

We now take up the proof of Theorem B in earnest.  Before investigating the continuous CM regularity of general vector bundles on $X,$ we will use Theorem A to compute the continuous CM-regularity of semihomogeneous bundles whose first Chern class is proportional to that of $\cO(1).$  Note that the following result completely describes the continuous CM regularity of semihomogeneous bundles when $X$ has Picard number 1.

\begin{prop}
\label{prop:pic-rk-1}
If $\cE$ is a semihomogeneous bundle on $X$ for which $c_{1}(\cE) \in {\Q}c_{1}(\cO(1)),$ then
\begin{equation}
\label{eq:semi-sharp}
{\rm reg}_{\rm cont}(\cE,\cO(1)) = \min\{m \in Z : \cE(m-g) \textnormal{ is nef}\} = \min\{m \in Z : \cE(m-g) \textnormal{ is a GV-sheaf}\}
\end{equation}
\end{prop}

\begin{proof}
The second equality in (\ref{eq:semi-sharp}) is an immediate consequence of Proposition \ref{prop:semihom-nef}; it therefore suffices to verify the first equality.  Since $c_{1}(\cE) \in {\Q}c_{1}(\cO(1))$ we have for all $m \in \Z$ that $\cE(m-g)$ is either of I.T. of index 0, has numerically trivial determinant, or is I.T. of index $g.$  Sheaves of the latter type are not GV-sheaves, so applying Proposition \ref{prop:semihom-nef} once more, we see that $\cE(m-g)$ is nef if and only if one of the first two possibilities holds.  The first equality in (\ref{eq:semi-sharp}) then follows from Theorem A.
\end{proof}

Together with Proposition \ref{prop:pic-rk-1}, the next result concludes the proof of Theorem B.

\begin{prop}
\label{prop:gen-ineq}
Let $X$ be an abelian variety of dimension $g \geq 1$ and let $\cE$ be a vector bundle on $X.$  Then we have that
\begin{equation}
\label{eq:thmb-ineq}
{\rm reg}_{\rm cont}(\cE,\cO(1)) \leq m_{\cE} := \min\{m \in \Z : \cE(m-g)\textnormal{ is a GV-sheaf}\}
\end{equation}
with equality if and only if $V^{0}(\cE^{\vee}(g+1-m_{\cE})) = \widehat{X}.$
\end{prop}

\begin{proof}
To see that the right-hand side exists and is finite, observe that by effective Serre vanishing (?), we have for $m,m' >> 0$ and all $\alpha \in \widehat{X}$ that $\cE^{\vee}(m+g) \otimes \alpha$ is globally generated and $\cE(m'-g) \otimes \alpha$ is globally generated with no higher cohomology.  For all such $m$ and $m'$ we have that $V^{g}(\cE(-m-g)) = V^{0}(\cE^{\vee}(m+g)) = \widehat{X}$ (so that $\cE(-m-g)$ is not a GV-sheaf) and $V^{i}(\cE(m'-g)) = \widehat{X}$ if and only if $i=0$ (so that $\cE(m'-g)$ is IT of index 0, hence a GV-sheaf).

If $m \in \Z$ satisfies the property that $\cE(m-g)$ is a GV-sheaf, then it follows from Proposition 3.1 of \cite{PP3} that $\cE(m-i)$ is I.T. of index 0 for $1 \leq i \leq g-1$.  We have at once that $V^{i}(\cE(m-i)) \neq \widehat{X}$ for all $i > 0$; this proves (\ref{eq:thmb-ineq}).  We have equality in (\ref{eq:thmb-ineq}) if and only if $V^{i}(\cE(m_{\cE}-1-i))=\widehat{X}$ for some $i > 0.$  Again using Proposition 3.1 of \cite{PP3}, we have that $V^{i}(\cE(m_{\cE}-1-i)) = \emptyset$ for $1 \leq i \leq g-1,$ so our condition is equivalent to $V^{g}(\cE(m_{\cE}-g-1)) = \widehat{X}$; by Serre duality it is in turn equivalent to $V^{0}(\cE^{\vee}(g+1-m_{\cE}))=\widehat{X}.$
\end{proof}

\begin{rmk}
Since semihomogeneous bundles are Gieseker-semistable with respect to any polarization, it would be interesting in light of Theorem A and Proposition \ref{prop:gen-ineq} to find a sharp numerical upper bound on $\crego{\cE}$ when $\cE$ is a vector bundle which is Gieseker-semistable with respect to $\cO(1)$.
\end{rmk}

\noindent
\textit{Proof of Proposition C:}  If $\cE$ is a vector bundle on $X$ which is a GV-sheaf, it is immediate that $m_{\cE} \leq g$; together with Proposition \ref{prop:gen-ineq}, this proves (\ref{eq:gv-ineq}).  Assume now that $\cE$ is semihomogeneous and $c_{1}(\cE^{\vee}(1))$ is ample.  By Proposition \ref{prop:semihom-pos} we have that $V^{0}(\cE^{\vee}(1))=\widehat{X},$ or equivalently $V^{g}(\cE(-1))=\widehat{X}$; it follows that $\crego{\cE} \geq g.$  Therefore $\crego{\cE} =g$ as desired. \hfill \qedsymbol 
\medskip

We end this section by showing that the conclusion of Proposition C can be strengthened if our vector bundles are taken to be $M-$regular. 

\begin{prop}
 	\label{prop:it-index}
 	Let $X$ be an abelian variety of dimension $g \geq 2$, let $\cO(1)$ be an ample and globally generated line bundle on $X,$ and let $\cF$ be a coherent sheaf on $X$ which is M-regular.  Then we have the following:
 	\begin{itemize}
 		\item[(i)]{${\rm reg}_{\rm cont}(\cF,\cO(1)) \leq {\rm reg}(\cF,\cO(1)) \leq g.$}
 		\item[(ii)]{If $\cF$ is locally free and $\cF^{\vee}(1)$ is continuously globally generated, then $${\rm reg}_{\rm cont}(\cF,\cO(1)) = {\rm reg}(\cF,\cO(1)) = g.$$}
 	\end{itemize}
 \end{prop}
 
 \begin{proof}
 	We first establish (i).  Let $\alpha \in \widehat{X}$ be given.  Since $\cO(1)$ is ample, $\O(g-i)$ is I.T. of index 0 for $1 \leq i \leq g-1.$  Given that $\cF$ is $M$-regular, Proposition \ref{prop:mreg-it} implies that $\cF(g-i) \otimes \alpha$ is I.T. of index 0 for $1 \leq i \leq g-1$.  Also, the $M$-regularity of $\cF$ implies that the cohomological support locus $V^{g}(\cF)$ is empty, so that $H^{g}(\cF \otimes \alpha)=0.$  Therefore ${\rm reg}(\cF \otimes \alpha,\cO(1)) \leq g$ for all $\alpha \in \widehat{X}.$  
	
	Setting $\alpha = \cO_{X},$ we have that ${\rm reg}(\cF,\cO(1)) \leq g.$  This concludes the proof of (i) since ${\rm reg}_{\rm cont}(\cF,\cO(1)) \leq {\rm reg}(\cF,\cO(1))$ by definition.
 	
 	Turning to (ii), it is enough by (i) to show that $\crego{\cF} = g$; we will be done once we verify that $V^{g}(\cF(-1))=\widehat{X}.$  By our hypothesis on $\cF^{\vee}(1),$ Proposition 3.1 of \cite{De} implies that there exists an abelian variety $Y$ of dimension $g$ and an isogeny $\pi : Y \to X$ such that $\pi^{\ast}(\cF^{\vee}(1) \otimes \alpha)$ is globally generated for all $\alpha \in \widehat{X}.$  In particular, we have that
 	\begin{equation}
 	\bigoplus_{{\xi} \in {\rm ker}(\pi^{\ast})} H^{0}(\cF^{\vee}(1) \otimes (\alpha \otimes \xi)) = H^{0}(\cF^{\vee}(1) \otimes \alpha \otimes \pi_{\ast}\cO_{Y}) \cong H^{0}(\pi^{\ast}(\cF^{\vee}(1) \otimes \alpha)) \neq 0
 	\end{equation} 
 	for all $\alpha \in \widehat{X}.$  Since ${\rm ker}(\pi^{\ast})$ is finite and the cohomological support loci of coherent sheaves on $X$ are algebraic subsets of $\widehat{X},$ there exists ${\xi}' \in {\rm ker}(\pi^{\ast})$ such that 
 	\begin{equation}
 	H^{g}(\cF(-1) \otimes \alpha \otimes {{\xi}'}^{\vee}) \cong H^{0}(\cF^{\vee}(1) \otimes \alpha^{\vee} \otimes {\xi}')^{\ast} \neq 0
 	\end{equation} for all $\alpha \in \widehat{X}.$  
 \end{proof}
 
 \begin{rem}
 	The $M-$regularity of $\cF$ in Proposition \ref{prop:it-index} cannot be weakened to the property of being a GV-sheaf, since ${\rm reg}(\cO,\cO(1))=g+1.$
 \end{rem}
 



\section{Verlinde Bundles}
\label{sec:verlinde}

In this section we apply our results to the Verlinde bundles, which were introduced in \cite{Po} and studied in \cite{O}.  Throughout, $X$ is the Jacobian of a smooth projective curve $C$ of genus $g \geq 1$ and $\Theta$ is a symmetric theta-divisor on $X.$  

\begin{defn}
Given a pair $(r,k)$ of positive integers, the \textit{Verlinde bundle} associated to $(r,k)$ is 
\begin{equation}
\mathbf{E}_{r,k} := {\rm det}_{\ast}\cO(k{\widetilde{\Theta}})
\end{equation}
where ${\rm det}: U_{C}(r,0) \to \widehat{X} = X$ is the determinant map associated to the moduli space $U_{C}(r,0)$ of semistable bundles of rank $r$ and degree $0$ on $C,$ and $\widetilde{\Theta}$ is the generalized theta-divisor on $U_{C}(r,0)$ associated to a theta-characteristic of $C.$
\end{defn}

\begin{prop}
\label{prop:verlinde}
If $(r,k)$ is a pair of positive integers whose greatest common divisor is odd, then $\mathbf{E}_{r,k}$ is ample, and for any $s \geq 2$ we have
\begin{equation}
\label{eq:verlinde}
{\rm reg}_{\rm cont}(\mathbf{E}_{r,k},\cO(s\Theta))={\Biggl\lceil}g-\frac{k}{rs}{\Biggr\rceil}
\end{equation}
\end{prop}

\begin{proof}
If $h={\rm gcd}(r,k),$ let $a:=r/h$ and $b:=k/h.$  By Remark 7.13 of \cite{Mu1} there exists a unique simple semihomogeneous bundle $\mathbf{W}_{a,b}$ which is symmetric and satisfies 
\begin{equation}
{\rm rank}(\mathbf{W}_{a,b}) = a^{g}, \hskip10pt {\rm det}(\mathbf{W}_{a,b}) \cong \cO(a^{g-1}b\Theta)
\end{equation}
Since $a^{g-1}b > 0$ we have from Proposition \ref{prop:semihom-pos} that $\mathbf{W}_{a,b}$ is ample.  Theorem 1 of \cite{O} implies that $\mathbf{E}_{r,k}$ is a direct sum of bundles of the form $\mathbf{W}_{a,b} \otimes \eta$, where $\eta$ is a torsion element of $\widehat{X};$ in particular $\mathbf{E}_{r,k}$ is ample as claimed.  We then have from Propositions \ref{prop:simp-uni} and \ref{prop:pic-rk-1} that
\begin{equation}
\label{eq:ver-w}
{\rm reg}_{\rm cont}(\mathbf{E}_{r,k},\cO(s\Theta)) = {\rm reg}_{\rm cont}(\mathbf{W}_{a,b},\cO(s\Theta)) = \min\{m : \cO(a^{g-1}(b+as(m-g))\Theta)\textnormal{ is nef}\}
\end{equation}
and (\ref{eq:verlinde}) follows from a routine calculation.  
\end{proof}

The next statement follows at once from (\ref{eq:verlinde}), Proposition \ref{prop:semihom-pos}, and (i) of Proposition \ref{prop:it-index}.
\begin{cor}
In the setting of Proposition \ref{prop:verlinde}, we have that
\begin{equation}
\label{eq:verlinde-eq}
{\Biggl\lceil}g-\frac{k}{rs}{\Biggr\rceil} \leq {\rm reg}(\mathbf{E}_{r,k},\cO(s\Theta)) \leq g
\end{equation}
In particular ${\rm reg}(\mathbf{E}_{r,k},\cO(s\Theta)) = g$ for $k < rs.$ \hfill \qedsymbol
\end{cor}

\section{Products}
\label{sec:products}

We end the paper by applying our results to products of abelian varieties. 

\begin{lem}
	\label{lem:semi-prod}
	Let $A_{1}, \ldots ,A_{m}$ be abelian varieties, and for $i=1, \cdots ,m$ let $\cF_i$ be a semihomogeneous bundle on $A_i.$  Then $\cF_{1} \boxtimes \cdots \boxtimes \cF_{m}$ is a semihomogeneous bundle on $X\deq A_{1} \times \cdots \times A_{m}.$
\end{lem}

\begin{proof}
	Let $x = (x_{1}, \cdots ,x_{m}) \in X$ be given, and for $i=1, \ldots ,m$ let $\alpha_{i} \in \widehat{A}_i$ satisfy $t_{x_i}^{\ast}\cF_i \cong \alpha_i \otimes \cF_i.$  
	If $p_{i} : X \to A_i$ is projection onto the $i$\textsuperscript{th}  factor, then we have
	\begin{equation*}
	t_{x}^{\ast}(\cF_1 \boxtimes \cdots \boxtimes \cF_{m}) \cong (p_1 \circ t_{x})^{\ast}\cF_1 \otimes \cdots \otimes (p_m \circ t_{x})^{\ast}\cF_m
	\end{equation*}
	\begin{equation*}
	\cong (t_{x_1} \circ p_1)^{\ast}\cF_1 \otimes \cdots \otimes (t_{x_m} \circ p_{m})^{\ast}\cF_{m} \cong (\alpha_{1} \boxtimes \cdots \boxtimes \alpha_{m}) \otimes (\cF_1 \boxtimes \cdots \boxtimes \cF_{m}).
	\end{equation*}
\end{proof}

\begin{rem}\label{rem:simplyfing}
	If in the setting of Lemma \ref{lem:semi-prod}, we assume that $A_{1},\ldots,A_{m}$ are pairwise non-isogenous, then for any line bundle $\cL$ on $X$ there exists for each $i$ a line bundle $\cL_{i}$ on $A_{i}$ such that $\cL\simeq \cL_1\boxtimes\ldots\boxtimes \cL_m$. 
\end{rem}

\begin{prop}
	\label{prop:non-isog-gen}
	Let $X = A_{1} \times A_{2},$ where $A_1$ and $A_2$ are non-isogenous abelian varieties each having dimension $g \geq 1,$ and let $\cO(1)$ be an ample and globally generated line bundle on $X.$  Then if $\cF$ is a semihomogeneous bundle on $X$ of rank $r \geq 1$ and $\cF_{1},\cF_{2}$ are the restrictions of $\cF$ to $A_1 \times \{0\},$ $\{0\} \times A_2$, respectively, we have
	\begin{equation*}
	{\rm reg}_{\rm cont}(\cF,\cO(1)) = {\rm reg}_{\rm cont}(\cF_{1} \boxtimes \cF_{2},\cO(1))\ ,
	\end{equation*}
\end{prop}

\begin{proof}
	As before, let $p_{i} : X \to A_{i}$ be projection onto the $i-$th factor.  Given that $\cF_{1} \boxtimes \cF_{2}$ is semihomogeneous by Lemma \ref{lem:semi-prod}, it is enough to show that
	\[
		c_{1}(\cF)/r \,=\, c_{1}(\cF_{1} \boxtimes \cF_{2})/r^{2}
	\]
	as elements of $N^1_\Q(X)$, thanks to Theorem A.  Since $c_{1}(\cF_{1} \boxtimes \cF_{2}) = r(p_{1}^{\ast}c_{1}(\cF_{1})+p_{2}^{\ast}c_{1}(\cF_{2})),$ we will be done once we show that $c_{1}(\cF) =  p_{1}^{\ast}c_{1}(\cF_{1})+p_{2}^{\ast}c_{1}(\cF_{2}).$ 
		 
	 For $i=1,2$ let $L_i$ be an ample line bundle on $A_i$ such that $c_{1}(L_{i})$ generates ${\rm Pic}(A_{i}),$ and let $D_{i} := p_{i}^{\ast}c_{1}(L_{i}).$  Then for $0 \leq k \leq 2g$ we have 
		\begin{equation}
		\label{eq:prod-int}
		D_{1}^{k} \cdot D_{2}^{2g-k} = \delta_{kg} \cdot L_{1}^{g} \cdot L_{2}^{g} 
		\end{equation}
	Since $A_{1}$ and $A_{2}$ are non-isogenous, for some $a_1, a_ 2 \in \Z$ for which $c_{1}(\cF) = a_{1}D_{1} + a_{2}D_{2}.$  Also, since the $A_{i}$ each have Picard number 1, there are $b_1, b_2 \in \Z$ such that $c_{1}(\cF_{i}) = b_{i}c_{1}(L_{i}).$  In particular, $p_{1}^{\ast}c_{1}(\cF_{1})+p_{2}^{\ast}c_{1}(\cF_{2}) = b_{1}D_{1}+b_{2}D_{2}.$  We have that
		\begin{equation}
		\label{eq:prod-comp}
		c_{1}(\cF) \cdot D_{1}^{g} = p_{2}^{\ast}c_{1}(\cF_{2}) \cdot D_{1}^{g} = b_{2}D_{1}^{g}D_{2}, \hskip20pt c_{1}(\cF) \cdot D_{2}^{g} = p_{1}^{\ast}c_{1}(\cF_{1}) \cdot D_{2}^{g} = b_{1}D_{1}D_{2}^{g}
		\end{equation}
	Combining this with (\ref{eq:prod-int}) yields
		\begin{equation}
		b_{1}D_{1}^{g} \cdot D_{2}^{g} = p_{1}^{\ast}c_{1}(\cF_{1}) \cdot D_{1}^{g-1}D_{2}^{g} = (a_{1}D_{1}+a_{2}D_{2}) \cdot D_{1}^{g-1} \cdot D_{2}^{g} = a_{2}D_{1}^{g} \cdot D_{2}^{g}
		\end{equation}
		\begin{equation*}
		b_{2}D_{1}^{g} \cdot D_{2}^{g} = p_{2}^{\ast}c_{1}(\cF_{2}) \cdot D_{1}^{g}D_{2}^{g-1} = (a_{1}D_{1}+a_{2}D_{2}) \cdot D_{1}^{g} \cdot D_{2}^{g-1} = a_{2}D_{1}^{g} \cdot D_{2}^{g}
		\end{equation*}
	Since $D_{1}^{g} \cdot D_{2}^{g} \neq 0$ by (\ref{eq:prod-int}), we conclude that $c_{1}(\cF) =  p_{1}^{\ast}c_{1}(\cF_{1})+p_{2}^{\ast}c_{1}(\cF_{2})$ as desired.

\end{proof}

Lastly, we illustrate the use of Proposition \ref{prop:non-isog-gen} with an explicit computation for the product of two non-isogenous elliptic curves. 

\begin{prop}
	\label{prop:non-isog-product}
	Let $X = E_1 \times E_2,$ where $E_1$ and $E_2$ are non-isogenous elliptic curves, and let $\cO(1) = \cO_{E_1}(1) \boxtimes \cO_{E_2}(1),$ where $\cO_{E_i}(1)$ is an ample and globally generated line bundle on $E_i$ for $i=1,2.$  Let $\cF$ be a semihomogeneous bundle on $X$ of rank $r \geq 1,$ and $m_{1}, m_{2}$ be the continuous CM-regularity of the restriction of $\cF$ to $E_{1} \times \{0\}, \{0\} \times E_{2}$ with respect to $\cO_{E_{1}}(1),\cO_{E_{2}}(1),$ respectively.  Then
	\begin{equation*}
	{\rm reg}_{\rm cont}(\cF,\cO(1)) = \max\{\min\{m_{1},m_{2}\}+1,\max\{m_{1},m_{2}\}\}\ .
	\end{equation*}
\end{prop}

\begin{proof}
	According to Proposition~\ref{prop:non-isog-gen} we can and will assume without loss of generality that $\cF = \cF_{1} \boxtimes \cF_{2}.$
	
	It follows from the K\"{u}nneth formula that $H^{1}(\cF(m-1))=0$ for $m \geq \max\{m_{1},m_{2}\}$ and that $H^{2}(\cF(m-2))=0$ for $m \geq \min\{m_{1},m_{2}\}+1.$  We then have that
	\begin{equation*}
	{\rm reg}_{\rm cont}(\cF,\cO(1)) \leq \max\{\min\{m_{1},m_{2}\}+1,\max\{m_{1},m_{2}\}\}
	\end{equation*}
	To sharpen this to equality, we consider two cases.  If $m_{1} < m_{2},$ then 
	\[
	\max\{\min\{m_{1},m_{2}\}+1,\max\{m_{1},m_{2}\}\} = m_{2}\ ,
	\]
	and we want to show that ${\rm reg}_{\rm cont}(\cF,\cO(1)) > m_{2}-1.$  
	
	In any case we have that
	\begin{equation*}
	h^{1}(\cF(m_{2}-2)) =  h^{1}(\cF_{1}(m_{2}-2)) \cdot h^{0}(\cF_{2}(m_{2}-2)) + h^{0}(\cF_{1}(m_{2}-2)) \cdot h^{1}(\cF_{2}(m_{2}-2))\ .
	\end{equation*}
	Since $m_{1} \leq m_{2}-1$ we have by Theorem~C that $\cF_{1}(m_{2}-2)$ is nef.  If it is ample, then 
	\[
	h^{0}(\cF_{1}(m_{2}-2)) \cdot h^{1}(\cF_{2}(m_{2}-2)) > 0
	\]
	and 
	\[
	{\rm reg}_{\rm cont}(\cF,\cO(1)) > m_{2}-1
	\]
	as desired. If it is nef but not ample, then by generic vanishing we have that $\cF_{1}(m_{2}-2)$ is semistable of slope $0$, so that 
	$\cF_{1}(m_{2}-3)$ is semistable of negative slope.  We then have that
	\begin{equation*}
	h^{2}(\cF(m_{2}-3)) = h^{1}(\cF_{1}(m_{2}-3)) \cdot h^{1}(\cF_{2}(m_{2}-3)) > 0\ ,
	\end{equation*}
	so that again ${\rm reg}_{\rm cont}(\cF,\cO(1)) > m_{2}-1$.
	
	It remains to settle the case $m_1 = m_2.$  Here 
	\[
	\max\{\min\{m_{1},m_{2}\}+1,\max\{m_{1},m_{2}\}\} = m_{1}+1 = m_{2}+1\ ,
	\]
	and we would like to show that ${\rm reg}_{\rm cont}(\cF,\cO(1)) > m_{1}.$  For $i=1,2$, Theorem~C yields  that $\cF_{i}^{\vee}(2-m_{i})$ is ample, so  
	\begin{equation*}
	h^{2}(\cF(m_{1}-2)) = h^{1}(\cF_{1}(m_{1}-2)) \cdot h^{1}(\cF_{2}(m_{2}-2)) = h^{0}(\cF_{1}^{\vee}(2-m_{1})) \cdot h^{0}(\cF_{2}(2-m_{2})) > 0
	\end{equation*}
	Thus ${\rm reg}_{\rm cont}(\cF,\cO(1)) > m_{1}$ as desired.  This concludes the proof.
\end{proof}


\begin{thebibliography}{10}

\begin{singlespace}
\bibitem[BPS]{BPS}
M. A. Barja, R. Pardini and L. Stoppino, {\sl Linear Systems on Irregular Varieties}, preprint, {\tt arXiv:1606.03290}.
\end{singlespace}


\begin{singlespace}
\bibitem[BL]{BL}
C. Birkenhake and H. Lange, {\sl Complex Abelian Varieties, 2nd ed.}, Springer-Verlag, New York 
\end{singlespace}

\begin{singlespace}
\bibitem[De]{De}
O. Debarre, {\sl On coverings of simple abelian varieties}, Bull. Math. soc. France \textbf{134} (2006), no. 2, p. 253-260.
\end{singlespace}

\begin{singlespace}
\bibitem[GL]{GL}
M. Green and R. Lazarsfeld, {\sl Deformation theory, generic vanishing theorems, and some conjectures of Enriques, Catanese and Beauville}, Invent. Math. \textbf{90} (1987), p.~ 389-407.
\end{singlespace}

\begin{singlespace}
\bibitem[Gr]{Gr}
N. Grieve, {\sl Index conditions and cup-product maps on Abelian varieties}, Internat. J. Math. \textbf{25} (2014), no. 4, 1450036.
\end{singlespace}
 
\begin{singlespace}
\bibitem[Gul]{Gul}
M. Gulbrandsen, {\sl Fourier-Mukai transforms of line bundles on derived equivalent abelian varieties}, Matematiche (Catania) \textbf{63} (2008), no. 1, p. 123--137.
\end{singlespace}

\begin{singlespace}
\bibitem[Ha]{Ha}
C. Hacon, {\sl A derived category approach to generic vanishing}, J. Reine Angew. Math. \textbf{575} (2004)
\end{singlespace}

\begin{singlespace}
\bibitem[Kol]{Kol}
J. Koll\'{a}r, {\sl Singularities of Pairs}, in {\sl Algebraic Geometry, Santa Cruz 1995} (1997), p. 221--287.
\end{singlespace}

\begin{singlespace}
\bibitem[KL]{KL}
A. K\"{u}ronya and V. Lozovanu, {\sl A Reider-type theorem for higher syzygies on abelian surfaces}, preprint, {\tt arxiv.org/1509.08621}.
\end{singlespace}

\begin{singlespace}
\bibitem[Mu1]{Mu1}
S. Mukai, {\sl Semi-homogeneous vector bundles on an Abelian variety}, J. Math. Kyoto Univ. \textbf{18} (1978), no. 2, p.~ 239-272.
\end{singlespace}

\begin{singlespace}
\bibitem[Mu2]{Mu2}
S. Mukai, {\sl Duality between ${\cD}(X)$ and ${\cD}(\widehat{X})$ with its application to Picard sheaves}, Nagoya Math. J. \textbf{81} (1981), p. 153--175.
\end{singlespace}

\begin{singlespace}
\bibitem[Mum]{Mum}
D. Mumford, {\sl Abelian varieties}, Tata Institute of Fundamental Research Studies in Mathematics, No. 5, Published for the Tata Institute of Fundamental Research, Bombay, 1970.
\end{singlespace}

\begin{singlespace}
\bibitem[Mus]{Mus}
Y. Mustopa, {\sl Castelnuovo-Mumford Regularity and GV-Sheaves on Irregular Varieties}, preprint, {\tt arxiv.org/1607.06550.}
\end{singlespace}

\begin{singlespace}
\bibitem[O]{O}
D. Oprea, {\sl The Verlinde bundles and the semihomogeneous Wirtinger duality}, J. Reine Angew. Math. \textbf{654} (2011), p. 181--217. 
\end{singlespace}
%

\begin{singlespace}
\bibitem[PP1]{PP1}
G. Pareschi and M. Popa, {\sl Regularity on abelian varieties I}, J. Amer. Math. Soc. \textbf{16} (2003), p.~ 285-302.
\end{singlespace}

\begin{singlespace}
\bibitem[PP2]{PP2}
G. Pareschi and M. Popa, {\sl GV-sheaves, Fourier-Mukai transform, and Generic Vanishing}, Amer. J. Math. \textbf{133} (2011), no.1, p.~ 235-271. 
\end{singlespace}

\begin{singlespace}
\bibitem[PP3]{PP3}
G. Pareschi and M. Popa, {\sl Regularity on abelian varieties III: relationship with generic vanishing and applications},  in  \textit{Grassmannians,  Moduli  Spaces  and  Vector  Bundles},
Clay  Mathematics  Proceedings \textbf{14} (2011), Amer. Math. Soc., Providence, RI, p.~ 141-167.
\end{singlespace}

\begin{singlespace}
\bibitem[Po]{Po}
M. Popa, {\sl Verlinde bundles and generalized theta linear series}, Trans. Amer. Math. Soc. \textbf{354} (2002), no. 5, p. 1869-1898.
\end{singlespace}

\begin{singlespace}
\bibitem[Y]{Y}
J.-H. Yang, {\sl Holomorphic Vector Bundles Over Complex Tori}, J. Korean Math. Soc. \textbf{26} (1989), no. 1, p.117-142.
\end{singlespace}

\end{thebibliography}
\end{document}